\documentclass[12pt]{article}
\usepackage{e-jc}
\usepackage{amsmath,amssymb,amsthm}

\usepackage{color}
\usepackage{diagbox}

\usepackage{microtype}

\theoremstyle{plain}
\newtheorem{theorem}{Theorem}
\newtheorem{lemma}[theorem]{Lemma}

\theoremstyle{definition}

\newtheorem{example}[theorem]{Example}

\theoremstyle{remark}

\newcommand{\F}{\mathbb{F}}
\newcommand{\Range}{{\rm Im}}

\author{Aleksandr Kodess\\
\small Department of Mathematics\\[-0.8ex]
\small University of Rhode Island\\[-0.8ex] 
\small Rhode Island, U.S.A.\\
\small\tt kodess@uri.edu\\
\and
Felix Lazebnik\thanks{Partially supported by NSF grant DMS-1106938-002}\\
\small Department of Mathematical Sciences\\[-0.8ex]
\small University of Delaware\\[-0.8ex]
\small Delaware, U.S.A.\\
\small\tt fellaz@udel.edu
}
\title{Connectivity of some algebraically defined digraphs}

\date{\dateline{Feb 20, 2015}{Aug 11, 2015}}
\begin{document}
\maketitle

\medskip
\centerline{{\sl Dedicated to the memory of Vasyl Dmytrenko (1961-2013)}}
\bigskip

\begin{abstract}
Let $p$ be a prime, $e$ a positive integer, $q = p^e$, and let
$\F_q$ denote the finite field of $q$ elements.
Let $f_i\colon\F_q^2\to\F_q$ be arbitrary functions, where $1\le i\le l$, $i$ and $l$
are  integers. The digraph $D = D(q;\bf{f})$, where
${\bf f}=(f_1,\dotso,f_l)\colon\F_q^2\to\F_q^l$, is defined as follows.
The vertex set of $D$
is $\F_q^{l+1}$. There is an arc from a vertex ${\bf x} = (x_1,\dotso,x_{l+1})$ to a vertex
${\bf y} = (y_1,\dotso,y_{l+1})$ if
$
x_i + y_i = f_{i-1}(x_1,y_1)
$
for all $i$, $2\le i \le l+1$.
In this paper we study the strong connectivity of $D$ and completely describe its strong components.
The digraphs $D$
are directed analogues of some algebraically defined graphs,  which have been studied extensively
and have many applications.

\end{abstract}

\section{Introduction and Results}
In this paper,
by a {\it directed graph} (or simply {\it digraph)}
$D$ we mean a pair $(V,A)$, where
$V=V(D)$ is the set of vertices and $A=A(D)\subseteq V\times V$ is the set of arcs.
The {\it order} of $D$ is the number of its vertices.
For an arc $(u,v)$, the first vertex $u$ is called its {\it tail} and the second
vertex $v$ is called its {\it head}; we denote such an arc by $u\to v$.
For an integer $k\ge 2$,  a {\it walk} $W$ {\it from} $x_1$ {\it to} $x_k$ in $D$ is an alternating sequence
$W = x_1 a_1 x_2 a_2 x_3\dots x_{k-1}a_{k-1}x_k$ of vertices $x_i\in V$ and arcs $a_j\in A$
such that the tail of $a_i$ is $x_i$ and the head of $a_i$ is $x_{i+1}$ for every
$i$, $1\le i\le k-1$.
Whenever the labels of the arcs of a walk are not important, we use the notation
$x_1\to x_2 \to \dotsb \to x_k$ for the walk.
In a digraph $D$, a vertex $y$ is {\it reachable} from a vertex $x$ if $D$ has a walk from $x$ to $y$. In
particular, a vertex is reachable from itself. A digraph $D$ is {\it strongly connected}
(or, just {\it strong}) if, for every pair $x,y$ of distinct vertices in $D$,
$y$ is reachable from $x$ and $x$ is reachable from $y$.
A {\it strong component} of a digraph $D$ is a maximal induced subdigraph of $D$ that is strong.
For all digraph terms not defined in this paper, see Bang-Jensen and Gutin \cite{Bang_Jensen_Gutin}.

Let $p$ be a prime, $e$ a positive integer, and $q = p^e$. Let
$\F_q$ denote the finite field of $q$ elements, and  $\F_q^*=\F_q\setminus\{0\}$.
We write $\F_q^n$ to denote the Cartesian product of $n$ copies of $\F_q$.
Let $f_i\colon\F_q^2\to\F_q$ be arbitrary functions, where $1\le i\le l$, $i$ and $l$
are positive integers. The digraph $D = D(q;f_1,\dotso,f_l)$, or just $D(q;\bf{f})$, where
${\bf f}=(f_1,\dotso,f_l)\colon\F_q^2\to\F_q^l$, is defined as follows.
(Throughout all of the paper the bold font is used to distinguish elements of
$\F_q^j$, $j\ge 2$, from those of $\F_q$, and we simplify the notation ${\bf f} ((x,y))$ and
${ f} ((x,y))$ to ${\bf f} (x,y)$ and
${ f} (x,y)$, respectively.)
The vertex set of $D$
is $\F_q^{l+1}$. There is an arc from a vertex ${\bf x} = (x_1,\dotso,x_{l+1})$ to a vertex
${\bf y} = (y_1,\dotso,y_{l+1})$ if and only if
\[
\label{incidence_condition}
x_i + y_i = f_{i-1}(x_1,y_1)\quad \mbox{for all }i,\ 2\le i \le l+1.
\]
We call the functions $f_i$, $1\le i\le l$,  the {\it defining functions} of $D(q;{\bf f})$.

If $l=1$ and ${\bf f}(x,y) = f_1(x,y) = x^m y^n$, $1\le m,n\le q-1$, we call  $D$ a {\it monomial} digraph, and denote it by  $D(q;m,n)$.

The digraphs $D(q; {\bf f})$
and $D(q;m,n)$ are directed analogues
of
some algebraically defined graphs,  which have been studied extensively
and have many applications. See
Lazebnik and Woldar \cite{LazWol01} and references therein; for some
subsequent work see  Viglione \cite{Viglione_thesis},
Lazebnik and Mubayi \cite{Lazebnik_Mubayi},
Lazebnik and Viglione \cite{Lazebnik_Viglione},
Lazebnik and Verstra\"ete \cite{Lazebnik_Verstraete},
Lazebnik and Thomason \cite{Lazebnik_Thomason},
 Dmytrenko, Lazebnik and Viglione \cite{DLV05},
 Dmytrenko, Lazebnik and Williford \cite{DLW07},
 Ustimenko \cite{Ust07}, Viglione \cite{VigDiam08},
 Terlep and Williford \cite{TerWil12}, Kronenthal \cite{Kron14},
 Cioab\u{a}, Lazebnik and Li \cite{CLL14},  and  Kodess \cite{Kod14}.

We note that $\F_q$ and $\F_q^l$ can be viewed as vector spaces over
$\F_p$ of dimensions $e$ and $el$, respectively. For $X\subseteq \F_q^l$, by $\langle X \rangle $
we denote the span of $X$ over $\F_p$, which is the set of all finite linear combinations of elements of $X$ with coefficients from $\F_p$.
For any vector subspace $W$ of $\F_q^l$, $\dim(W)$ denotes the dimension of $W$ over $\F_p$.
If $X \subseteq \F_q^l$, let ${\bf v} + X = \{ {\bf v} + {\bf x}\colon {\bf x}\in X\}$.
Finally,  let   $\Range({\bf f}) = \{(f_1(x,y),\dotso,f_l(x,y))\colon(x,y)\in\F_q^2\}$
denote the image of function ${\bf f}$.

In this paper we study strong connectivity of $D(q;{\bf f})$.
We mention that
by Lagrange's interpolation (see, for example, Lidl, Niederreiter \cite{Lidl_Niederreiter}),
each $f_i$ can be uniquely represented by
a bivariate polynomial of degree at most $q-1$ in each of the variables. We therefore also call
 functions $f_i$ {\it defining polynomials}.

In order to state our results,  we need the following notation.
For every ${\bf f}\colon \F_q^2\to\F_q^l$,
	we define
	\[
	{\bf g}(t) = {\bf f}(t,0) - {\bf f}(0,0),
	\quad
	{\bf h}(t) = {\bf f}(0,t) - {\bf f}(0,0),
 \]
\[
	\tilde{{\bf f}}(x,y) = {\bf f}(x,y) - {\bf g}(y) - {\bf h}(x),
	\]
\[
{\bf f_0}(x,y) = {\bf f}(x,y) - {\bf f}(0,0), \quad \text{and}
\]
\[
{\bf \tilde{f}_0}(x,y) = {\bf f_0}(x,y) - {\bf g}(y) - {\bf h}(x).
\]

As ${\bf g}(0) = {\bf h}(0) ={\bf 0}$,  one can view the coordinate function  $g_i$ of ${\bf g}$
(respectively, $h_i$ of ${\bf h}$), $i=1,\ldots, l$,  as the sum of all terms of
the polynomial $f_i$ containing only indeterminate $x$ (respectively, $y$),
and having zero constant term.
We, however, wish to emphasise that  
in the definition of
$\tilde{{\bf f}}(x,y)$, ${\bf g}$ is evaluated at $y$, and ${\bf h}$ at $x$.
Also, we will often write a vector $(v_1, v_2, \ldots, v_{l+1})\in \F_q^{l+1} = V(D)$
as an ordered pair $(v_1, {\bf v}) \in \F_q \times \F_q^l$,
where ${\bf v}= (v_2, \ldots, v_{l+1})$.
\bigskip

The main result of this paper is the following theorem, which gives necessary and sufficient
conditions for the strong connectivity of $D(q;{\bf f})$ and provides a description of
its strong components in terms of $\langle \Range({\bf \tilde{f}_0})\rangle$ over $\F_p$.

\begin{theorem}
	\label{main_key_theorem_on_D(q;f)}
Let $D= D(q;{\bf f})$,
$D_0= D(q;{\bf f_0})$, 
$W_0 = \langle \Range({\bf \tilde{f}_0})\rangle$ over $\F_p$, and
$d= \dim (W_0)$ over $\F_p$.
    Then the following statements hold.
	\begin{itemize}
		\item[(i)]
        If $q$ is odd, then the digraphs $D$ and $D_0$ are isomorphic.
		Furthermore, the vertex set of the strong component of $D_0$ containing a vertex $(u,{\bf v})$ is
		$$\label{thm18_part_i}
		\Bigl\{
		(a,{\bf v} + {\bf h}(a) - {\bf g}(u) + W_0)
		\colon
		a\in\F_q
		\Bigr\}
		\cup
		\Bigl\{
		(b,-{\bf v}+{\bf h}(b) + {\bf g}(u) + W_0)
        \, 
		\colon
		b\in\F_q
		\Bigr\} =$$
        \begin{equation}\label{cosets}
         \Bigl\{
		(a,{\bf \pm v} + {\bf h}(a) \mp {\bf g}(u) + W_0)\Bigl\}.
		\end{equation}

		The vertex set of the strong component of $D$ containing a vertex $(u,{\bf v})$ is
        \begin{equation}
        \label{cosets_D}
		\Bigl\{
		(a,{\bf v} + {\bf h}(a) - {\bf g}(u) +  W_0)
		\colon
		a\in\F_q
		\Bigr\}
		\cup
		\Bigl\{
		(b,-{\bf v}+{\bf h}(b) + {\bf g}(u) + {\bf f}(0,0) + W_0)
        \, 
		\colon
		b\in\F_q
		\Bigr\}.
        \end{equation}
		In particular, $D\cong D_0$ is strong if and only if
		$W_0 = \F_q^l$
		or, equivalently,
		$d = el$.

        If $q$ is even, then the strong component of $D$ containing a vertex
        $(u,{\bf v})$ is
        \begin{equation}
        \label{comp_descr_even}
        \Bigl\{
        (a, {\bf v} + {\bf h}(a) + {\bf g}(u) + W_0)\colon a\in\F_q
        \Bigr\}
        \cup
        \Bigl\{
        (a, {\bf v} + {\bf h}(a) + {\bf g}(u) + {\bf f}(0,0) + W_0)\colon a\in\F_q
        \Bigr\}
        \end{equation}
        \[
        =
        \Bigl\{
        (a, {\bf v} + {\bf h}(a) + {\bf g}(u) + W)\colon a\in\F_q
        \Bigr\},
        \]
        where $W = W_0 + \langle \{f(0,0)\}\rangle = \langle \Range({\bf \tilde{f}})\rangle$.
        \item[(ii)]
        If $q$ is odd, then $D\cong D_0$ has  $(p^{el-d}+1)/2$
		strong components. One of them is of order $p^{e+d}$. All
		other $(p^{el-d}-1)/2$ strong components are isomorphic, and each is of order $2p^{e+d}$.
        \smallbreak
        If $q$ is even, then the number of strong components in $D$ is
        $2^{el-d}$, provided ${\bf f}(0,0)\in W_0$, and 
        it is $2^{el-d-1}$ otherwise.
        In each case, all strong components are isomorphic,  and  are of orders 
        $2^{e+d}$ and $2^{e+d+1}$, respectively. 
%
%
	\end{itemize}
\end{theorem}

We note here that for $q$ even the digraphs $D$ and $D_0$
are generally not isomorphic.

We apply this theorem to monomial digraphs $D(q;m,n)$. For these digraphs we can restate the connectivity results more explicitly.

\begin{theorem}
	\label{monomial_strong}
    Let $D=D(q; m,n)$ and let $d= (q-1, m, n)$ be the greatest common divisor of $q-1$,  $m$ and $n$.
    For each positive divisor $e_i$ of $e$, let  $q_i
	:= (q-1)/(p^{e_i} - 1)$, and let $q_s$ be the largest of the $q_i$ that divides $d$.
	Then the following statements hold.
	\begin{itemize}
		\item[(i)] The vertex set of the strong component of $D$ containing a vertex $(u,v)$ is
		\begin{align}
		\label{comp_vertex_set}
		\{(x, v + \F_{p^{e_s}})\colon \, x\in\F_q\}\cup
		\{(x,-v + \F_{p^{e_s}})\colon \, x\in\F_q\}.
		\end{align}
		In particular, $D$ is strong if and only if $q_s = 1$ or, equivalently, $e_s=e$.
		\item[(ii)] 
        If $q$ is odd, then $D$ has  $(p^{e-e_s}+1)/2$
		strong components. One of them
is of order $p^{e+e_s}$. All
		other $(p^{e-e_s}-1)/2$ strong components are all isomorphic and each is of order $2p^{e+e_s}$.
        \smallbreak
If $q$ is even, then $D$ has $2^{e-e_s}$ strong components,
        all isomorphic, and each is of order $2^{e+e_s}$.

	\end{itemize}
\end{theorem}

\bigskip

Our proof of Theorem \ref{main_key_theorem_on_D(q;f)} is presented
in  Section \ref{sect2},  and the proof of Theorem \ref{monomial_strong} is in
Section \ref{sect3}. In Section \ref{sec4} we suggest two areas for further investigation.

\section{Connectivity of $D(q;{\bf f})$} \label{sect2}

Theorem \ref{main_key_theorem_on_D(q;f)} and our proof below were inspired by the ideas
from \cite{Viglione_thesis}, where the components of similarly defined
bipartite simple graphs were described.


\bigskip

We now prove Theorem \ref{main_key_theorem_on_D(q;f)}.
\begin{proof}
Let $q$ be odd.
We first show that $D \cong D_0$.
The map
$\phi\colon V(D)\to V(D_0)$ given by
\begin{equation}
\label{isomor_phi}
(x,{\bf y})\mapsto (x,{\bf y}-\frac{1}{2}{\bf f}(0,0))
\end{equation}
is clearly a bijection.
We check that $\phi$ preserves adjacency.
Assume that $((x_1,{\bf x}_2),(y_1,{\bf y}_2))$ is an arc in $D$, that is,
${\bf x}_2 + {\bf y}_2 = {\bf f}(x_1,y_1)$.
Then, since $\phi((x_1,{\bf x_2})) = (x_1,{\bf x}_2 - \frac{1}{2}{\bf f}(0,0))$ and
$\phi((y_1,{\bf y_2})) = (y_1,{\bf y}_2 - \frac{1}{2}{\bf f}(0,0))$, we have
\[
({\bf x}_2 - \frac{1}{2}{\bf f}(0,0))
+
({\bf y}_2 - \frac{1}{2}{\bf f}(0,0))
=
{\bf f}(x_1,y_1) - {\bf f}(0,0)
=
{\bf f}_0(x_1,y_1),
\]
and so $(\phi((x_1,{\bf x}_2)),\phi((y_1,{\bf y}_2)))$ is an arc in $D_0$.
As the above steps are reversible, $\phi$ preserves non-adjacency
as well.
Thus, $D(q;{\bf f}) \cong D(q;{\bf f_0})$.
\bigskip

We now obtain the description (\ref{cosets})
of the strong components of
$D_0$, and then explain how
the description (\ref{cosets_D}) of the strong components of
$D$ follows from
(\ref{cosets}).

Note that as ${\bf f_0}(0,0) = {\bf 0}$,  we have ${\bf g}(t)= {\bf f_0}(t,0)$, ${\bf h}(t)= {\bf f_0}(0,t)$,
${\bf g}(0)= {\bf h}(0) = {\bf 0}$,  and
${\bf \tilde{f}_0}(x,y) = {\bf f_0}(x,y) - {\bf g}(y) - {\bf h}(x)$.

Let
$\tilde{\alpha}_1,\dotso,\tilde{\alpha}_d \in  \Range({\bf \tilde{f}_0})$ be a basis for
$W_0$. 
Now, choose $x_i,y_i\in\F_q$ be such that ${\bf \tilde{f}_0}(x_i,y_i) = \tilde{\alpha}_i$, $1\le i\le d$.


Let $(u,{\bf v})$ be a vertex of $D_0$.
We first show that a vertex  $(a, {\bf v} + {\bf y})$  is reachable
from $(u,{\bf v})$ if  ${\bf y}\in {\bf h}(a) - {\bf g}(u)+W_0$.
In order to do this, we write an arbitrary ${\bf y}\in {\bf h}(a) -{\bf g}(u)+ W_0$  as
\[
{\bf y}
=
{\bf h}(a)-{\bf g}(u)
+
(a_1\tilde{\alpha}_1 + \dotsb + a_d\tilde{\alpha}_d),
\]
for some $a_1,\dotso,a_d\in\F_p$,  and consider the following directed walk in $D_0$:
\begin{align}
(u,{\bf v})    &   \to (0,-{\bf v}+{\bf f_0}(u,0)) = (0,-{\bf v}+{\bf g}(u))\nonumber\\
&\to (0,{\bf v} - {\bf g}(u))\label{step11}\\
&\to (x_1,-{\bf v}+{\bf g}(u) + {\bf f_0}(0,x_1)) = (x_1,-{\bf v}+{\bf g}(u) + {\bf h}(x_1))\\
&\to (y_1,{\bf v}-{\bf g}(u)-{\bf h}(x_1)+{\bf f_0}(x_1,y_1))\\
&\to (0,-{\bf v}+{\bf g}(u)+{\bf h}(x_1)-{\bf f_0}(x_1,y_1)+{\bf g}(y_1))\\
& = (0,-{\bf v}+{\bf g}(u)-{\bf \tilde{f}_0}(x_1,y_1)) = (0,-{\bf v}+{\bf g}(u)-\tilde{\alpha}_1)\\
&\to (0,{\bf v}-{\bf g}(u)+\tilde{\alpha}_1))\label{steplast}.
\end{align}
Traveling through vertices whose first coordinates are $0$, $x_1$, $y_1$, $0$, $0$, and $0$ again
(steps \ref{step11}--\ref{steplast})
as many times as needed,
one can reach vertex $(0,{\bf v}-{\bf g}(u) +a_1\tilde{\alpha}_1)$.
Continuing a similar walk through vertices whose first coordinates
are $0$, $x_i$, $y_i$, $0$, $0$, and $0$, $2\le i \le d$,
as many times as needed,
one can reach vertex $(0,{\bf v}-{\bf g}(u) + (a_1\tilde{\alpha}_1 + \ldots + a_i\tilde{\alpha}_i))$,
and so on, until the  vertex
$(0, - {\bf v}+{\bf g}(u)-(a_1\tilde{\alpha}_1+\dotsb+a_d\tilde{\alpha}_d))$ is reached.
The vertex $(a, {\bf v} + {\bf y})$ will be its out-neighbor. Here we indicate just some of the vertices along this
path:
\begin{align*}
&\to \dotso\\
&\to (0,{\bf v}-{\bf g}(u)+ a_1\tilde{\alpha}_1)\\
&\to (x_2,-{\bf v}+{\bf g}(u) - a_1\tilde{\alpha}_1+{\bf h}(x_2))\\
&\to (y_2,{\bf v}-{\bf g}(u)+a_1\tilde{\alpha}_1-{\bf h}(x_2)+{\bf f_0}(x_2,y_2))\\
&\to (0,-{\bf v}+{\bf g}(u)-a_1\tilde{\alpha}_1+{\bf h}(x_2)-{\bf f_0}(x_2,y_2)+{\bf g}(y_2))\\
& = (0,-{\bf v}+{\bf g}(u)-a_1\tilde{\alpha}_1-\tilde{\alpha}_2)\\
&\to (0,{\bf v}-{\bf g}(u)+a_1\tilde{\alpha}_1+\tilde{\alpha}_2)\\
&\to \dotso\\
& = (0, -{\bf v}+{\bf g}(u)- a_1\tilde{\alpha}_1-a_2\tilde{\alpha}_2)\\
&\to \dotso\\
& = (0, - {\bf v}+{\bf g}(u)-(a_1\tilde{\alpha}_1+\dotsb+a_d\tilde{\alpha}_d))\\
&\to (a, {\bf v} - {\bf g}(u) + {\bf h}(a)+(a_1\tilde{\alpha}_1+\dotsb+a_d\tilde{\alpha}_d))\\
&= (a,{\bf v} + {\bf y}).
\end{align*}
Hence,  $(a,{\bf v} + {\bf y})$ is reachable from $(u,{\bf v})$ for any $a\in\F_q$
and any
${\bf y}\in {\bf h}(a)-{\bf g}(u)+ W_0$,
as claimed.
A slight modification of this argument shows that $(a,-{\bf v} + {\bf y})$
is reachable from $(u,{\bf v})$
for any
${\bf y}\in {\bf h}(a)+{\bf g}(u)+ W_0$.

Let us now explain that every vertex of $D_0$ reachable from  $(u,{\bf v})$ is in the set
$$\{ (a, \pm {\bf v} \mp {\bf g}(u) + {\bf h}(a) + W_0)\colon  \; a\in \F_q \}.$$
We will need  the following identities on $\F_q$ and $\F_q^2$, respectively,  which can be checked easily
using  the definition of ${\bf \tilde{f}}$:
\begin{align*}
&{\bf \tilde{f}_0}(t,0)= {\bf g}(t) - {\bf h}(t) = - {\bf \tilde{f}_0}(0,t)\;\; \text{and}\\
&{\bf {f_0}}(x,y) =  {\bf g}(x) + {\bf h}(y) + {\bf {\tilde{f}_0}}(x,y)
- {\bf\tilde{f}_0}(0,y) +   {\bf\tilde{f}_0}(0,x).
\end{align*}
The identities immediately imply that for every $t, x, y \in \F_q$,
\begin{align*}
&{\bf g}(t) - {\bf h}(t) \in W_0\;\; \text{and}\\
&{\bf {f_0}}(x,y) =  {\bf g}(x) + {\bf h}(y) + w\;\;\text{for some}\;\; w=w(x,y)\in W_0.
\end{align*}
Consider a  path with  $k$  arcs, where $k > 0$ and even,  from $(u, {\bf v})$ to $(a,
 {\bf v}+ {\bf y})$:
$$(u, {\bf v}) = (x_0, {\bf v})\to (x_1, \ldots)\to (x_2, \ldots) \to \cdots \to (x_k, {\bf v}+{\bf y})
= (a, {\bf v}+ {\bf y}).$$
Using the definition of an arc  in $D_0$,  and setting ${\bf f_0}(x_i, x_{i+1}) =
{\bf g}(x_i) +{\bf h}(x_{i+1}) +w_i$, and ${\bf g}(x_i) - {\bf h}(x_i)=w_i'$, with all
$w_i, w_i'\in W_0$, we obtain:
\begin{align*}
{\bf y} &= {\bf f_0}(x_{k-1},x_{k}) - {\bf f_0}(x_{k-2},x_{k-1}) + \cdots +
{\bf f_0}(x_{1},x_{2}) - {\bf f_0}(x_{0},x_{1}) \\
&= \sum_{i=0}^{k-1} (-1)^{i+1}{\bf f_0}(x_i, x_{i+1}) =
\sum_{i=0}^{k-1} (-1)^{i+1}({\bf g}(x_i) +{\bf h}(x_{i+1}) +w_i) \\
&= -{\bf g}(x_0) +{\bf h}(x_{k}) + \sum_{i=1}^{k-1} (-1)^{i-1}({\bf g}(x_{i}) - {\bf h} (x_i)) +
\sum_{i=0}^{k-1} (-1)^{i+1}w_i\\
&=   -{\bf g}(x_0) +{\bf h}(x_{k}) + \sum_{i=1}^{k-1} (-1)^{i-1}w_i' +
\sum_{i=0}^{k-1} (-1)^{i+1} w_i.
\end{align*}
Hence, ${\bf y} \in   -{\bf g}(x_0) +{\bf h}(x_{k}) + W_0$.
Similarly,  for any path
$$(u, {\bf v}) = (x_0, {\bf v})\to (x_1, \ldots)\to (x_2, \ldots) \to \cdots \to (x_k, {\bf v}+{\bf y})
= (a, -{\bf v}+ {\bf y}),$$
with  $k$  arcs, where $k$ is odd and at least 1, we obtain ${\bf y} \in   {\bf g}(x_0) +{\bf h}(x_{k}) + W_0$.

The digraph $D_0$ is strong if and only if $W_0= \langle \Range({\bf \tilde{f}_0})\rangle = \F_q^l$ or,
equivalently, $d = el$. Hence part (i) of the theorem is proven for $D_0$ and $q$ odd.
\bigskip

Let $(u,{\bf v})$ be an arbitrary vertex of a strong component of $D$. 
The image of this vertex under the isomorphism $\phi$, defined in (\ref{isomor_phi}), is
$(u,{\bf v}-\frac{1}{2}{\bf f}(0,0))$, which belongs to the strong component 
of $D_0$ whose description is given by (\ref{cosets}) with ${\bf v}$
replaced by ${\bf v}-\frac{1}{2}{\bf f}(0,0)$.
Applying
the inverse of $\phi$ to each vertex of this component of $D_0$ immediately yields the description of
the component of $D$ given by (\ref{cosets_D}).
This establishes the validity of part (i) of Theorem \ref{main_key_theorem_on_D(q;f)} for $q$ odd.


\bigskip

For $q$ even we first apply an argument similar to the one we used above for establishing components of $D_0$ for $q$ odd. As $p=2$, the argument becomes much shorter, and we obtain (\ref{comp_descr_even}).  Then we note that if
\[
(u,{\bf v}) = (x_0,{\bf v})\to (x_1,\dotso)\to (x_2,\dotso) \to\dotsb\to (x_k,{\bf v}+ {\bf y})
\]
is a path in $D$, then $${\bf y} = \sum_{i=0}^{k-1} {\bf f_0}(x_i,x_{i+1}) + \delta \cdot{\bf f}(0,0),$$ where $\delta = 1$ if $k$ is odd,  and  $\delta = 0$ if $k$ is even.
\bigskip

For (ii), we first recall that any two cosets
of $W_0$ in $\F_p^{kl}$ are disjoint or coincide. It is clear that for $q$ odd,  the cosets (\ref{cosets})
 coincide if and only if ${\bf v}\in {\bf g}(u) + W_0$. The vertex set of this strong component is
$\{(a,{\bf h}(a)+ W_0)\colon a\in\F_q \}$, which
shows that this is the unique component of such type.  As  $|W_0|= p^d$,
 the component contains $q\cdot p^d= p^{e+d}$  vertices. In all other cases the cosets are disjoint, and their union  is of order $2q p^d= 2p^{e+d}$.
 Therefore the number of strong components  of $D_0$, which is isomorphic to $D$, is
\[
\frac{|V(D)| - p^{e+d}}{2p^{e+d}} +1= \frac{p^{e(l+1)}-p^{e+d}}{2p^{e+d}}+1 = \frac{p^{el-d}+1}{2}.
\]
For $q$ even, our count follows the same ideas as for $q$ odd, and the formulas giving the number of strongly connected components and the order of each component follow from 
(\ref{comp_descr_even}).
\bigskip

For the isomorphism of strong components of the same order, let $q$ be odd, and
let $D_1$ and $D_2$ be two distinct strong components of $D_0$ each of order $2p^{e+d}$. Then
there exist
$(u_1,{\bf v}_1),(u_2,{\bf v}_2)\in V(D_0)$ with
${\bf v}_1\not\in {\bf g}(u_1)+W_0$
and
${\bf v}_2\not\in {\bf g}(u_2) + W_0$ such that
$V(D_1)=\{(a,{\bf v}_1 + {\bf h}(a) - {\bf g}(u_1) + W_0)\colon \, a\in \F_q\}$
and
$V(D_2)=\{(a,{\bf v}_2 + {\bf h}(a) - {\bf g}(u_2) + W_0)\colon \, a\in \F_q\}$.

Consider a map $\psi: V(D_1)\to V(D_2)$ defined by
\[
(a,\pm{\bf v}_1 + {\bf h}(a)\mp{\bf g}(u_1) + {\bf y})
\mapsto
(a,\pm{\bf v}_2 + {\bf h}(a)\mp{\bf g}(u_2) + {\bf y}),
\]
for any $a\in\F_q$ and any ${\bf y}\in W_0$.
Clearly, $\psi$ is a bijection. Consider an arc $(\alpha, \beta)$ in $D_1$.
If $\alpha=(a,{\bf v}_1+{\bf h}(a)-{\bf g}(u_1)+{\bf y})$, then
$\beta=(b,-{\bf v}_1-{\bf h}(a)+{\bf g}(u_1)-{\bf y} + {\bf f_0}(a,b))$ for some
$b\in\F_q$. Let us
check that $(\psi (\alpha), \psi(\beta))$ is an arc in $D_2$. In order to find an expression for
the second coordinate of $\psi (\beta)$,  we first rewrite the second coordinate of $\beta$ as
$-{\bf v}_1+{\bf h}(a)+{\bf g}(u_1)+{\bf y'}$,  where  ${\bf y'} \in W_0$.  In order to do this, we use
the definition of
${\bf \tilde{f}_0}$ and the obvious equality  ${\bf g} (b) - {\bf h} (b) =
{\bf \tilde{f}_0}(b,0)\in W_0$.
So we have:
\begin{align*}
&-{\bf v}_1-{\bf h}(a)+{\bf g}(u_1)-{\bf y}+{\bf f}(a,b)\\
=&-{\bf v}_1-{\bf h}(a)+{\bf g}(u_1)-{\bf y}+{\bf \tilde{f}_0}(a,b)+{\bf g}(b)+{\bf h}(a)\\
=&-{\bf v}_1+{\bf h}(b)+{\bf g}(u_1)+ ({\bf g}(b)-{\bf h}(b))-{\bf y}+{\bf \tilde{f}_0}(a,b)\\
=&-{\bf v}_1+{\bf h}(b)+{\bf g}(u_1)+{\bf y'},
\end{align*}
where ${\bf y'}=  ({\bf g}(b)-{\bf h}(b))-{\bf y}+{\bf \tilde{f}_0}(a,b) \in W_0$.
Now it is clear that $\psi (\alpha) = (a,{\bf v}_2+{\bf h}(a)-{\bf g}(u_2)+{\bf y})$ and
$\psi (\beta) =
(b,-{\bf v}_2+{\bf h}(b)+{\bf g}(u_2)+ {\bf y'})$
are the tail and the head of an
arc in $D_2$.
Hence $\psi$ is an isomorphism
of digraphs $D_1$ and $D_2$.

An argument  for the isomorphism of  all strong components for $q$ even is absolutely
similar. This ends the proof of the theorem.
\end{proof}

\bigskip

We illustrate Theorem \ref{main_key_theorem_on_D(q;f)} by the following example.

\begin{example}
	Let $p\ge 3$ be prime, $q = p^2$, and $\F_q \cong \F_p(\xi)$, where $\xi$ is a primitive
	element in $\F_q$.
	Let us define $\emph{}f\colon\F_q^2\to\F_q$ by the following table:
\bigskip

	\begin{center}
		\begin{tabular}{|c|ccc|}
			\hline
			\diagbox{$y$}{$x$}    &    0        &    \;1    &    $x\neq 0,1$    \\
			\hline
			0          &    0        &    \;$\xi$                         &    1\\
			1   &    $\xi$                         &   \; $2\xi$    &    $\xi$\\
			$y\neq 0,1$                        &    2                     &   \; $\xi$     &    0\\
			\hline
		\end{tabular}
		.
	\end{center}
	As $1$ and $\xi$ are values of $f$,  $\langle \Range(f) \rangle = \F_q^2$.
 Nevertheless,  $D(q; f)$ is not strong as we show below.
\bigskip

	In this example, since $l=1$, the function ${\bf f} = f$. Since $f(0,0) = 0$, $f_0=f$, and
\[
	{\bf g}(t)=
	g(t) = f(t,0)=
	\begin{cases}
	0,    & t = 0,\\
	\xi,  & t = 1,\\
	1,    & \mbox{otherwise}
	\end{cases}
	,\quad
	{\bf h}(t)=
	h(t) = f(0,t)=
	\begin{cases}
	0,    & t = 0,\\
	\xi,  & t = 1,\\
	2,    & \mbox{otherwise}
	\end{cases}
	.
	\]

The function ${\bf \tilde{f}_0}(x,y) = \tilde{f} (x,y) =
f(x,y) - {f}(y,0) -  f(0,x)$
	can be represented by the table
\bigskip

	\begin{center}
		\begin{tabular}{|c|ccc|}
			\hline
			\diagbox{$y$}{$x$}    &    0        &    $ \; 1 $    &    $x\neq 0,1$   \\
			\hline
			0                        &    0  &   \; 0   &  -1\\
			1   &    0  &   \; 0   &  -2\\
			$y\neq 0,1$                        &    1  &   \; -1   &  -3\\
			\hline
		\end{tabular}
		,
	\end{center}	
	\bigskip
and so $\langle \Range(\tilde{f}_0) \rangle = \F_p \neq  \langle \Range({f}) \rangle=\F_{p^2}$.

As $l=1$, $e=2$, and $d=1$,
		$D(q;f)$ has $(p^{le-d}+1)/2 = (p+1)/2$ strong components. For $p=5$, there are three of them.   If  $\F_{25} = \F_5 [\xi]$,  where $\xi$ is a root of
	 $X^2+4X+2\in\F_5[X]$,    these components  can be  presented as:
	\[
	\{
	(a,h(a)+\F_5)\colon a\in\F_{25}
	\}
	,
	\]
	\[
	\{
	(a,h(a)-\xi+\F_5)
	\colon
	a\in\F_{25}
	\}
	\cup
	\{
	(b,h(b)+\xi+\F_5)
	\colon
	b\in\F_{25}
	\}
	,
	\]
	\[
	\{
	(a,h(a)+2\xi+\F_5)
	\colon
	a\in\F_{25}
	\}
	\cup
	\{
	(b,h(b)-2\xi+\F_5)
	\colon
	b\in\F_{25}
	\}
	.
	\]
\end{example}

\section{Connectivity of $D(q, m,n)$}\label{sect3}

The goal of this section is to prove Theorem \ref{monomial_strong}.

For any $t\ge 2$ and integers $a_1, \ldots, a_t$, not all zero, let $(a_1,\ldots, a_t)$
(respectively $[a_1,\ldots, a_t]$) denote
the greatest common divisor (respectively, the least common multiple) of these numbers.
Moreover, for an integer $a$, let  $\overline{a} = (q-1, a)$.
Let $<\xi>= \F_q^*$, i.e., $\xi$ is a generator of the cyclic group $\F_q^*$.
(Note the difference between $< \cdot >$ and $\langle \cdot \rangle$ in our notation.)
Suppose   $A_k= \{x^k: x\in \F_q^*\}$, $k\ge 1$.  It is well known (and easy to show) that
$A_k = < \xi^{\overline k} > $ and $|A_k|=(q-1)/{\overline k}$.

We recall that for each positive divisor $e_i$ of $e$,   $q_i = (q-1)/(p^{e_i} - 1)$.

\begin{lemma}
\label{smallest_field_Feps}
Let $q_s$ be the largest of the $q_i$ dividing $\overline k$. Then $\F_{p^{e_s}}$ is
the smallest subfield of $\F_q$ in which $A_k$ is contained. Moreover,
 $\langle A_k \rangle = \F_{p^{e_s}}$.
\end{lemma}
\begin{proof}
By definition of ${\overline k}$, $q_s$ divides $k$, so $k = tq_s$ for some integer $t$. Thus
for any $x\in\F_q$,
\[
x^k = x^{tq_s} = \Bigl(x^\frac{p^e-1}{p^{e_s}-1}\Bigr)^t \in \F_{p^{e_s}},
\]
as  $x^{(p^e-1)/(p^{e_s}-1)}$  is the norm of $x$ over $\F_{p^{e_s}}$ and hence is in $\F_{p^{e_s}}$.
  Suppose now
that $A_k \subseteq \F_{p^{e_i}}$, where $e_i < e_s$.
Since $A_k$ is a subgroup of $\F_{p^{e_i}}^*$, we have that $|A_k|$ divides $|\F_{p^{e_i}}^*|$,
that is,
$(q-1)/{\overline k}$  divides $p^{e_i} - 1$.
Then ${\overline k} = r \cdot (q - 1)/(p^{e_i} - 1) = rq_i$  for
some integer $r$. Hence,  $q_i$ divides ${\overline k}$,
and a contradiction is obtained as $q_i > q_s$. This proves that
$\langle A_k \rangle$ is a subfield of $\F_{p^{e_s}}$ not contained in any smaller
subfield of $\F_q$. Thus $\langle A_k \rangle = \F_{p^{e_s}}$.
%
\end{proof}
\vspace{0.25cm}

Let $A_{m,n}=
\{ x^m y^n\colon  x,y\in\F_q^*\}$, $m,n\ge 1$. Then, obviously,  $A_{m,n}$ is a subgroup of $\F_q^*$,
and $A_{m,n} = A_m A_n$ -- the product of subgroups $A_m$ and  $A_n$.
\begin{lemma}
\label{lemma_size_Amn}
Let $d= (q-1,m,n)$.  Then $A_{m,n} = A_d$.
\end{lemma}

\begin{proof}
As $A_m$ and $A_n$ are subgroups of $\F_q^*$, we have
\begin{equation}
\label{Amn_order}
|A_{m,n} | = | A_m A_n | = \frac{| A_m| | A_n |}{|A_m \cap A_n|}.
\end{equation}

It is well known (and easy to show) that
if $x$ is a generator of a
cyclic group, then for any integers $a$ and $b$,
$< x^a >\cap < x^b > =
<x^{[a,b]}> $. Therefore,
$A_m \cap A_n =  \, < \xi^{[\overline{m},\, {\overline n}]} >$ and
 $|A_m \cap A_n| =  (q-1)/\overline{[\overline{m},\, {\overline n}]}$.

We wish to show that $|A_{m,n}| = |A_d|$, and since in a cyclic
group any two subgroups of equal order are equal, that would imply $A_{m,n}=A_d$.

From (\ref{Amn_order}) we find
\begin{equation}\label{Amn}
|A_{m,n}|=  \frac{ (q-1)/{\overline m}\, \cdot (q-1)/{\overline n} }
{(q-1)/\overline{[\overline{m},\, {\overline n}]}}=
\frac{(q-1) \cdot \overline{[\overline{m},\, {\overline n}]} }  {{\overline m} \cdot  {\overline n}}.
\end{equation}
We wish to  simplify the last fraction  in (\ref{Amn}).
Let $M$ and $N$ be such that $q-1 = M\overline{m} = N\overline{n}$. As
$d = (q-1, m,n) = (\overline{m},\overline{n})$, we have
$\overline{m} = dm'$ and $\overline{n} = dn'$ for some co-prime integers $m'$ and $n'$.
 Then $q-1 = dm'M = dn'N$ and $(q-1)/d = m'M = n'N$. As
$(m',n')=1$, we have $M = n't$ and $N = m't$ for some integer $t$. This implies that
$q-1 = dm'n't$. For any integers $a$ and $b$, both nonzero, it holds that $[a,b]= ab/(a,b)$.
Therefore, we have
\[
[\overline{m},\overline{n}] = [dm',dn']
 =
\frac{dm'dn'}{(dm',dn')} = \frac{dm'dn'}{d(m',n')} = dm'n' .
\]
Hence,
 $\overline{[\overline{m},\, {\overline n}]}  =
 (q-1,[\overline{m},\overline{n}]) = (dm'n't, dm'n') = dm'n'$,  and
\[
|A_{m,n}| = \frac{ (q-1) \cdot dm'n'}{\overline{m} \cdot \overline{n}} =
\frac{ (q-1) \cdot dm'n'}{dm' \cdot dn'} =\frac{q-1}{d}.
\]

Since $\overline{d} =(q-1,d) = d$ and $|A_d|= (q-1)/\overline{d}$,  we have
$|A_{m,n} |= |A_d|$ and so $A_{m,n} = A_d$.
\end{proof}

We are ready to prove Theorem \ref{monomial_strong}.
\begin{proof}
For $D=D(q;m,n)$, we have
$$\langle \Range{({\bf \tilde{f}_0})}\rangle = \langle \Range(f) \rangle =
\langle \Range ({x^my^n}) \rangle =
\langle A_{m,n} \rangle = \langle A_d \rangle = \F_{p^{e_s}},$$ where the last two equalities are due
to Lemma \ref{lemma_size_Amn} and Lemma \ref{smallest_field_Feps}.

Part (i) follows immediately from applying Theorem \ref{main_key_theorem_on_D(q;f)}
with $W = \F_{p^{e_s}}$, ${\bf g} = {\bf h} = 0$.
Also, $D$ is strong if and only if $\F_{p^{e_s}} = \F_q$, that is, if and only if
$e_s = e$, which is equivalent to $q_s = 1$.

The other statements of Theorem \ref{monomial_strong} follow directly from the corresponding
parts of Theorem \ref{main_key_theorem_on_D(q;f)}.
\end{proof}

\section{Open problems}\label{sec4}

We would like to conclude this paper with two suggestions for further investigation.
\bigskip

\noindent{\bf Problem 1.}  Suppose the digraphs $D(q; {\bf f})$ and $D(q; m,n)$ are strong.
What are their diameters?
\bigskip

\noindent{\bf Problem 2.}  Study the connectivity of graphs $D(\F; {\bf f})$, where
${\bf f}\colon \F^2\to \F^l$,  and $\F$ is a finite  extension of
      the field $\mathbb{Q}$ of rational numbers.
 \bigskip

\section{Acknowledgement}    The authors are thankful to the anonymous referees
whose thoughtful comments improved the paper;
to Jason Williford for pointing to a mistake in
the original version of Theorem \ref{main_key_theorem_on_D(q;f)}; and to
William Kinnersley for carefully reading the paper and pointing to a number of small errors.

\end{document}